\documentclass[a4paper]{article}

\usepackage{amsmath,amsfonts,amssymb,amsthm}
\usepackage{thmtools,thm-restate}
\usepackage{pstricks,caption}

\newtheorem{lemma}{Lemma}
\newtheorem{corollary}[lemma]{Corollary}

\newtheorem{proposition}[lemma]{Proposition}

\theoremstyle{definition}
\newtheorem{definition}[lemma]{Definition}

\newcommand{\Sphere}{\mathbb{S}^3}  
\newcommand{\crcle}{\mathbb{S}^1}  
\newcommand{\nhd}{\mathcal{N}} 
\DeclareMathOperator{\ms}{MS} 
\DeclareMathOperator{\dist}{d} 
\DeclareMathOperator{\lk}{lk} 

\begin{document}

\title{The quasi-isomorphism class of the Kakimizu complex}
\author{Jessica E. Banks}
\date{}
\maketitle
\begin{abstract}
It has been shown that the Kakimizu complex of a knot is quasi-isomorphic to $\mathbb{Z}^n$ for some $n\geq 0$.
We give a lower bound on $n$, matching the upper bound previously given.
\end{abstract}


A \textit{Seifert surface} for a knot $K$ in $\Sphere$ is a compact, connected, orientable surface whose boundary is $K$. We consider Seifert surfaces up to ambient isotopy in the knot exterior $E=\Sphere\setminus\nhd(K)$. The \textit{Kakimizu complex} $\ms(K)$ of $K$ is a simplicial complex that records the structure of the set of minimal genus Seifert surfaces for $K$. The vertices are given by the isotopy classes of minimal genus Seifert surfaces for $K$, and distinct vertices span a simplex if the vertices can be realised disjointly in $E$. The Kakimizu complex of the unknot is a single vertex; we will assume in this paper that $K$ is not the unknot (the result is immediate in this case).

It is known that, given $K$, there is an upper bound on the dimension of any simplex in $\ms(K)$. If $K$ is either a torus knot or hyperbolic then $\ms(K)$ has only finitely many vertices, but if $K$ is a satellite knot then $\ms(K)$ may be infinite (\cite{MR1177053}) and even locally-infinite (\cite{zbMATH05899716}). In addition, Przytycki and Schultens have shown that $\ms(K)$ is contractible (\cite{MR2869183}).

In \cite{zbMATH06369090}, Johnson, Pelayo and Wilson proved that $\ms(K)$ is quasi-Euclidean. That is, there exists $n\in\mathbb{N}\cup\{0\}$ such that $\ms(K)$ is quasi-isomorphic to $\mathbb{Z}^n$. Here the metric on $\ms(K)^1$ is the graph metric where each edge has length $1$. The authors give an upper bound on $n$, and suggest that this is also a lower bound. Our aim is to show that this is indeed a lower bound. To do so, we must recall the key elements of their proof.

\medskip

Consider an incompressible torus $T$ properly embedded in $E$. In $\Sphere$, it must bound a solid torus on one side, and this solid torus necessarily contains $K$. We will describe the solid torus as being `inside' $T$, and the knot-complement component of $E\setminus T$ as being `outside' $T$.

We next need to consider the JSJ decomposition of $E$ (see, for example, \cite{zbMATH05141213}).
Choose a minimal collection $T_1,\ldots,T_N$ of incompressible tori, pairwise disjoint, such that the complement of $\bigcup T_i$ consists of Seifert fibered pieces and atoroidal pieces. Let $E_0,\ldots,E_N$ be the (closures of the) regions of $E\setminus\bigcup T_i$. We may arrange that $E_0$ meets $\nhd(K)$, and, for $1\leq j\leq N$, that $E_j$ has $T_j$ as one of its boundary components with $E_j$ lying outside $T_j$. Following the terminology of \cite{zbMATH06369090}, we will refer to each $E_j$ as a \textit{block}. Set $T_0$ to be the torus $\partial E=\partial\nhd(K)$. Then $T_0$ is also incompressible in $E$ and is
 a boundary component of $E_0$, with $E_0$ outside of $T_0$.

Johnson, Pelayo and Wilson define the \textit{core} of $E$ to be the union of the core blocks, where $E_j$ is a \textit{core block} if every (minimal genus) Seifert surface for $K$ intersects $E_j$. Equivalently, $E_j$ is a core block if $K$ is homologically non-trivial in the solid torus $V_j$ inside $T_j$. Note that the core of $E$ is connected and contains $E_0$. 
The following result shows that each $T_j$ contained in the interior of the core of $K$ has a preferred slope. For each $j$, let $K_j$ be the core of $V_j$, and view $V_j$ as a neighbourhood of $K_j$.

\begin{proposition}[\cite{zbMATH06369090} Proposition 2]\label{slopeprop}
If $E_j$ is a core block then there is a slope $\alpha_j$ on $T_j$ such that, if $R$ is any minimal genus Seifert surface for $K$, every curve of $R\cap T_j$ that is essential in $T_j$ (of which there is at least one) is parallel to $\alpha_j$. Moreover, $\alpha_j$ is the longitude of $V_j$ (that is, $\alpha_j$ is the boundary of a Seifert surface for $K_j$, as an unoriented curve).
\end{proposition}

Although the slope $\alpha_j$ is determined by the knot $K_j$ in $\Sphere$, the number and orientation of the curves $R\cap T_j$ are controlled by the position of the surface $R$. 
 More precisely, $[R\cap T_j]=a_j[\alpha_j]=[K]$ in $H_1(V_j;\mathbb{Z})$ for some $a_j\in\mathbb{Z}$. Here $|a_j|$ is equal to the winding number of $K$ in $V_j$.

\subsection*{A group action}

Johnson, Pelayo and Wilson make use of an action of $\mathbb{Z}^N$ on $\ms(K)$. For a fixed $k$, choose a product neighbourhood $W_k$ of $T_k$ in $E$. Choose a product structure on $W_k$, expressing it as $\crcle\times\crcle\times I$, where the first $\crcle$ corresponds to the slope $\alpha_k$ on $T_k$.
Define $\phi_{T_k}\colon E\to E$ by
\[
\phi_{T_k}(x)=
\begin{cases}
x & x\notin W_k,\\
(z,e^{i(\theta +2\pi t)},t) & x=(z,e^{i\theta})\in W_k.
\end{cases}
\]
Note that if the product neighbourhoods $W_i$ are disjoint then these homeomorphisms of $E$ act independently.

We then define the action $\Phi\colon \mathbb{Z}^N\times\ms(K)\to\ms(K)$ by
\[
\Phi(r_1,\ldots,r_N,R)=\phi_{T_1}^{r_1}\circ\cdots\circ\phi_{T_N}^{r_N}(R).
\]
The moral of \cite{zbMATH06369090} is that all infinite directions in $\ms(K)$ come from `spinning around the tori' using this action. The upper bound on the quasi-dimension of $\ms(K)$ comes from counting the number of different ways of spinning around tori like this. To give a lower bound, we will fix a minimal genus Seifert surface $R_K$ for $K$, and show that acting on $R_K$ gives enough Seifert surfaces that are different (and distant in $\ms(K)$) from each other.

\subsection*{Basepoint} 

Choose a minimal genus Seifert surface $R_K$ for $K$, for use as a reference point. We will now edit $R_K$ to suit our purposes, but continue to denote it by $R_K$. For this we will use the following two results, which, although not explicitly stated, make up the proof of \cite{burde2013knots} Lemma 16.3.

\begin{lemma}
Let $K_S$ be a (satellite) knot, and let $T_S$ be an essential torus in the complement of $K_S$. Let $R_S$ be a minimal genus Seifert surface for $K_S$, in general position with respect to $T_S$. Then it is possible, by surgering along subdics and subannuli of $T_S$ and discarding closed components, to change $R_S$ to a minimal genus Seifert surface $R'_S$ such that all components of $R'_S\cap T_S$ are parallel and oriented in the same direction.
\end{lemma}

\begin{lemma}
Let $K_S$ be a knot in $\Sphere$, and let $R_S$ be a connected, oriented surface properly embedded in the exterior of $K_S$, such that all boundary components of $R_S$ are longitudes of $K_S$ oriented in the same direction.
Then $R_S$ has at most one boundary component.
\end{lemma}

Choose a JSJ torus $T_k$. We can edit $R_K$ so that all components of $R_K\cap T_k$ are parallel and oriented in the same direction. If $T_k$ does not lie in the interior of the core then $R_K$ is disjoint from $T_k$.
On the other hand, if $T_k$ lies in the interior of the core then
these curves are parallel to $\alpha_k$ and there are $|a_k|$ of them. Then $R_K\setminus V_k$ is formed of $|a_k|$ minimal genus Seifert surfaces for $K_k$. We may then replace these with $|a_k|$ parallel copies of a single one of those components. Note that these changes all take place in $W_k$ or outside $T_k$, without affecting anything further inside.

By working in this way inductively outwards from $K$, we achieve the following result.

\begin{lemma}
We may choose $R_K$ such that, for each $j$, all curves of $R_K\cap T_j$ are parallel to $\alpha_j$ and oriented the same way, and all components of $R_K\cap E_j$ are parallel to each other with a single boundary component each on $T_j$.
\end{lemma}

Some readers may find it helpful to picture the surface $R_K$ we have just constructed in terms of branched surfaces. We will not explicitly use this viewpoint in this paper.

\subsection*{Fibred blocks}

The upper bound on the dimension of $\ms(K)$ given in \cite{zbMATH06369090} depends on the number of core blocks that are fibred. 
For each core block $E_j$, we can ask whether a connected component $R_j$ of $R_K\cap E_j$ is a fibre for $E_j$ (that is, whether the complement of $R_j$ in $E_j$ is $R_j\times I$). Note that the answer to this question is determined only by the curves $R_K\cap\bigcup T_i$ (which depend only on $K$), and is not dependent on the specific choice of surface $R_K$. It is possible that a block might be fibred with a different `boundary pattern', but we are not interested in such cases in this paper. The upper bound in \cite{zbMATH06369090} is one less that the number of core blocks that are not fibred. Denote this number by $N'$. 
Our aim is to construct a quasi-isometric embedding of $\mathbb{Z}^{N'}$ into $\ms(K)$.

The intuitive explanation for this value is that we can spin $R_K$ around each torus it intersects, but spinning around $T_0$ can be reversed by isotopy, and if $E_j$ is fibred then spinning around $T_j$ has the same effect, up to isotopy, as spinning around each of the other boundary components of $E_j$.

For the purposes of our proof, we will need to forget about some of the tori $T_0,\ldots,T_N$, according to which ones we will use for spinning around. For convenience, we will re-label the objects we are considering.
Starting with the list $T_0,\ldots,T_N$, remove each $T_j$ such that $E_j$ is not a core block. Also remove $T_0$. If $E_j$ is fibred for $j\geq 1$ then remove $T_j$. Finally, if $E_0$ is fibred then remove one remaining torus that is now `innermost', in the sense that it is not separated from $K$ by any of the other remaining tori.
Re-label the remaining list of tori as $T'_1,\ldots,T'_{N'}$ and their neighbourhoods as $W'_1,\ldots,W'_{N'}$. For convenience, write $W'=\bigcup W'_i$.
Also label the regions of $E\setminus\bigcup T'_i$ as $E'_0,\ldots,E'_{N'}$. As before we may arrange that $T'_j$ is a boundary component of $E'_j$, with $E'_j$ lying outside $T'_j$. The advantage of our new notation is that $T'_j\cap R_K\neq\emptyset$ for each $j$, and no $E'_j$ is fibred.

We are now ready to define our quasi-isometric embedding using the group action $\Phi$.
Note that, in defining each $\Phi_j$, we had some choice in the product structure on $W_j$. 
For notational convenience, we will assume that the product structure on each $W'_j$ has been chosen such that, moving from the inside of $W'_j$ to outside, $\phi_{T'_j}$ twists in the direction given by the orientation on the Seifert surface $R_K$. Figure \ref{pic1} illustrates this convention for the torus $T'_1$.
\begin{figure}[htbp]
\centering
\psset{xunit=.5pt,yunit=.5pt,runit=.5pt}
\begin{pspicture}(640,350)
{
\pscustom[linestyle=none,fillstyle=solid,fillcolor=gray]
{
\newpath
\moveto(260,209.99999738)
\curveto(260,143.7258274)(206.27416998,89.99999738)(140,89.99999738)
\curveto(73.72583002,89.99999738)(20,143.7258274)(20,209.99999738)
\curveto(20,276.27416736)(73.72583002,329.99999738)(140,329.99999738)
\curveto(206.27416998,329.99999738)(260,276.27416736)(260,209.99999738)
\closepath
\moveto(620,209.99999477)
\curveto(620,143.72582479)(566.27416998,89.99999477)(500,89.99999477)
\curveto(433.72583002,89.99999477)(380,143.72582479)(380,209.99999477)
\curveto(380,276.27416475)(433.72583002,329.99999477)(500,329.99999477)
\curveto(566.27416998,329.99999477)(620,276.27416475)(620,209.99999477)
\closepath
}
}
{
\pscustom[linestyle=none,fillstyle=solid,fillcolor=white]
{
\newpath
\moveto(180,209.99999738)
\curveto(180,187.90860739)(162.09138999,169.99999738)(140,169.99999738)
\curveto(117.90861001,169.99999738)(100,187.90860739)(100,209.99999738)
\curveto(100,232.09138738)(117.90861001,249.99999738)(140,249.99999738)
\curveto(162.09138999,249.99999738)(180,232.09138738)(180,209.99999738)
\closepath
\moveto(540,209.99999477)
\curveto(540,187.90860477)(522.09138999,169.99999477)(500,169.99999477)
\curveto(477.90861001,169.99999477)(460,187.90860477)(460,209.99999477)
\curveto(460,232.09138476)(477.90861001,249.99999477)(500,249.99999477)
\curveto(522.09138999,249.99999477)(540,232.09138476)(540,209.99999477)
\closepath
}
}
{
\pscustom[linewidth=1,linecolor=black]
{
\newpath
\moveto(220,209.99999738)
\curveto(220,165.8172174)(184.18277999,129.99999738)(140,129.99999738)
\curveto(95.81722001,129.99999738)(60,165.8172174)(60,209.99999738)
\curveto(60,254.18277737)(95.81722001,289.99999738)(140,289.99999738)
\curveto(184.18277999,289.99999738)(220,254.18277737)(220,209.99999738)
\closepath
\moveto(580,209.99999477)
\curveto(580,165.81721478)(544.18277999,129.99999477)(500,129.99999477)
\curveto(455.81722001,129.99999477)(420,165.81721478)(420,209.99999477)
\curveto(420,254.18277475)(455.81722001,289.99999477)(500,289.99999477)
\curveto(544.18277999,289.99999477)(580,254.18277475)(580,209.99999477)
\closepath
}
}
{
\pscustom[linewidth=3,linecolor=black]
{
\newpath
\moveto(140,20)
\lineto(140,210)
\moveto(500,210)
\lineto(499.95879606,170.02546779)
\curveto(470.68611524,167.60219056)(447.88889432,193.72891971)(446.68601544,221.5625415)
\curveto(444.98971321,260.81357136)(480.19706247,290.80271289)(517.8086043,291.37707252)
\curveto(568.20201186,292.1466219)(606.22160889,246.70501303)(605.95357704,198.13561245)
\curveto(605.63369543,140.17064155)(556.38287607,94.49214281)(500.08059909,90.26016773)
\lineto(500,20)
}
}
{
\pscustom[linewidth=2,linecolor=black]
{
\newpath
\moveto(290,210)
\lineto(350,210)
\moveto(140,50)
\lineto(160,50)
\moveto(500,50)
\lineto(520,50)
\moveto(140,200)
\lineto(160,200)
\moveto(500,200)
\lineto(520,200)
}
}
{
\pscustom[linewidth=1,linecolor=black,fillstyle=solid,fillcolor=black]
{
\newpath
\moveto(330,210)
\lineto(322,202)
\lineto(350,210)
\lineto(322,218)
\lineto(330,210)
\closepath
\moveto(150,50)
\lineto(146,46)
\lineto(160,50)
\lineto(146,54)
\lineto(150,50)
\closepath
\moveto(510,50)
\lineto(506,46)
\lineto(520,50)
\lineto(506,54)
\lineto(510,50)
\closepath
\moveto(150,200)
\lineto(146,196)
\lineto(160,200)
\lineto(146,204)
\lineto(150,200)
\closepath
\moveto(510,200)
\lineto(506,196)
\lineto(520,200)
\lineto(506,204)
\lineto(510,200)
\closepath
}
}
{
\put(50,80){$E'_0$}
\put(110,210){$E'_1$}
\put(290,230){$\phi_{T'_1}$}
}
\end{pspicture}
\caption{\label{pic1}}
\end{figure}

We define $\Theta\colon\mathbb{Z}^{N'}\to\ms(K)$ by
\[
\Theta(r_1,\ldots,r_{N'})=(\phi_{T'_1})^{5r_1}\circ\cdots\circ(\phi_{T'_{N'}})^{5r_{N'}}(R_K).
\]
That is, up to re-labelling, $\Theta$ is the restriction of $\Phi^5$ to the coordinates corresponding to the tori $T'_1,\ldots,T'_{N'}$ and the surface $R_K$. The use of the power $5$ here is not significant; its purpose is to remove the need to consider `small cases' later.

\subsection*{Distances}

To show that $\Theta$ is a quasi-isometric embedding, we need to calculate distances in $\ms(K)$. The distance $\dist_{\ms(K)}$ between two vertices in $\ms(K)$ is defined using the graph metric where each edge has length $1$. In \cite{MR1177053}, Kakimizu gave a method for calculating the distance using the infinite cyclic cover of $E$ corresponding to the kernel of the linking number $\lk\colon\pi_1(E)\to\mathbb{Z}$.

Choose a minimal genus Seifert surface $R$ for $K$. We can build the infinite cyclic cover $\widetilde{E}$ of $E$ as follows. Let $E_R$ be $E$ cut along the surface $R$. Then the boundary of $E$ is divided into three parts: two copies of $R$, which can be distinguished using the orientation of $R$, and an annulus that is the torus $\partial E$ cut along the simple closed curve $\partial R$. To form $\widetilde{E}$, stack countably many copies of $E_R$ by gluing the positive side of $R$ in the $n$th copy of $\partial E_R$ to the negative side of $R$ in the $(n+1)$th copy. The quotient map is given by mapping each copy of $E_R$ to $E_R$ by the identity, then taking the quotient map from $E_R$ to $E$. The covering transformation is given by translating along the line of copies of $E_R$.

Now choose a second minimal genus Seifert surface $R'$ for $K$ that is not isotopic to $R$. We can calculate the distance between vertices $R$ and $R'$ in $\ms(K)$ as follows.
Choose a lift $\widetilde{R'}$ of $R'$ to $\widetilde{E}$. Isotope $\widetilde{R'}$ within $\widetilde{E}$ to minimise the number, $d$, of copies of $E_R$ that it intersects. Then $\dist_{\ms(K)}(R,R')=d$. Note that $d=1$ if and only if $\widetilde{R'}$ can be isotoped to be disjoint from all lifts of $R$ in $\widetilde{E}$, which is as we would expect given the definition of adjacency in $\ms(K)$.

The difficult part of using this criterion is establishing when $\widetilde{R'}$ has been suitably positioned. The following result, which has its roots in work of Waldhausen, enables us to verify this by only considering the position of the surface $R'$ relative to $R$ within $E$. This version is restricted to the case of knots in $\Sphere$ (the original was for use in more general manifolds).

\begin{definition}
Let $S$ be a compact, connected, orientable surface, and let $\rho$ be a finite (possibly empty, possibly disconnected) submanifold of $\partial S$. Let $M_S$ be the manifold given by taking $S\times I$ and identifying $\{x\}\times I$ to a point for each $x\in\rho$. We call any manifold of this form a \textit{product region}.

We say that the surfaces $R$ and $R'$ \textit{bound a product region} if there exists a product region $M_S$ of this form properly embedded in (the closure of) $E\setminus(R\cup R')$ such that $M_S\cap R=S\times\{0\}$ and $M_S\cap R'=S\times\{1\}$.
\end{definition}

This definition should be viewed as the three-dimensional analogue of when two arcs or curves in a surface `bound a bigon'.

\begin{proposition}[\cite{MR2869183} Proposition 3.2]\label{prop:productregions}
If $R$ and $R'$ intersect transversely and do not bound a product region then $R$ and $R'$ realise $\dist_{\ms(K)}(R,R')$.
\end{proposition}

In other words, if we can arrange that $R$ and $R'$ are transverse and $E\setminus(R\cup R')$ does not include any product regions, then we can count the distance between $R$ and $R'$ without needing to consider any further isotopy of $R'$ (or equivalently of $\widetilde{R'}$). This is the technique we will use to verify that the images of points under $\Theta$ are suitably far apart in $\ms(K)$.

Note that if $M_S$ is a product region between $R$ and $R'$, the intersection $M_S\cap R$ is a connected, orientable surface. Thus if a component of $E\setminus(R\cup R')$ 
meets $R$ on both the positive and the negative sides then this component is not a product region between $R$ and $R'$.

\subsection*{Proof}

\begin{proposition}
The map $\Theta$ is a quasi-isometric embedding of $\mathbb{Z}^{N'}$ into $\ms(K)$.
\end{proposition}
\begin{proof}
Let $(r_1,\ldots,r_{N'}),(s_1,\ldots,s_{N'})\in\mathbb{Z}^{N'}$.
Using the action $\Phi$, we may assume without loss of generality that $(r_1,\ldots,r_{N'})=(0,\ldots,0)$. With this assumption, $\Theta(r_1,\ldots,r_{N'})=R_K$. We may also assume that $(s_1,\ldots,s_{N'})\neq(0,\ldots,0)$, which implies that $\max(|s_1|,\ldots,|s_{N'}|)>0$. 

Denote by $S$ a copy of $\Theta(s_1,\ldots,s_{N'})$. We will position $S$ carefully with respect to $R_K$, show that there are no product regions bounded by $R_K$ and $S$, and read off a lower bound on $\dist_{\ms(L)}(R_K,S)$.
If there is a value of $k$ such that $s_k=0$ then the torus $T'_k$ plays no part in this process. We should therefore forget about $T'_k$, as we have already forgotten about some of the other $T_i$. Rather than re-labelling the tori and complementary regions again, we will instead assume that $s_k\neq 0$ for each $k$. This does not impact on the method of proof; it is simply for notational convenience.

We can think of the surface $S$ as being divided into different pieces. In $E'_j\setminus W'$, $R_K$ and $S$ coincide, and are made up of $|a_j|$ parallel copies of the same connected surface. Meanwhile, each component of $S\cap W'_j$ is an annulus that winds $|s_j|$ times around $T_j$ relative to $R_K$.
We will re-position $S$ by considering these pieces separately.

First consider $W'_k$ for some $k$. Each of $R_K\cap W'_k$ and $S\cap W'_k$ consists of parallel annuli properly embedded in $W'_k$. Picture the case where $R_K\cap W'_k$ is a single annulus $A_K$ and $S\cap W'_k$ is a single annulus $A_S$. Note that initially $\partial A_K=\partial A_S$. Because $A_S$ winds around $W'_k$ at least once relative to $A_K$, there is a well-defined choice of direction to isotope each boundary component of $A_S$ within a neighbourhood of $\partial A_K$ to make $\partial A_S$ and $\partial A_K$ disjoint without otherwise affecting $A_K\cap A_S$ (see Figure \ref{pic2}a).
\begin{figure}[htbp]
\centering
(a)
\psset{xunit=.5pt,yunit=.5pt,runit=.5pt}
\begin{pspicture}(640,280)
{
\pscustom[linestyle=none,fillstyle=solid,fillcolor=gray]
{
\newpath
\moveto(260,139.99999738)
\curveto(260,73.7258274)(206.27416998,19.99999738)(140,19.99999738)
\curveto(73.72583002,19.99999738)(20,73.7258274)(20,139.99999738)
\curveto(20,206.27416736)(73.72583002,259.99999738)(140,259.99999738)
\curveto(206.27416998,259.99999738)(260,206.27416736)(260,139.99999738)
\closepath
\moveto(620,139.99999477)
\curveto(620,73.72582479)(566.27416998,19.99999477)(500,19.99999477)
\curveto(433.72583002,19.99999477)(380,73.72582479)(380,139.99999477)
\curveto(380,206.27416475)(433.72583002,259.99999477)(500,259.99999477)
\curveto(566.27416998,259.99999477)(620,206.27416475)(620,139.99999477)
\closepath
}
}
{
\pscustom[linewidth=6,linecolor=black]
{
\newpath
\moveto(140,20)
\lineto(140,100)
\moveto(500,20)
\lineto(500,100)
}
}
{
\pscustom[linewidth=3,linecolor=black]
{
\newpath
\moveto(489.57888277,101.72394496)
\curveto(461.73196707,106.23586212)(446.86197761,136.06919021)(452.20065017,162.16160768)
\curveto(459.96577269,200.11314268)(501.05128353,219.95028115)(537.00504017,211.47997483)
\curveto(586.53505491,199.81125268)(612.04247945,145.68815249)(599.78531104,98.3895)
\curveto(590.55473928,62.77004928)(562.49287717,34.48398742)(527.98194956,22.58086182)
\moveto(139.87819606,99.76529779)
\curveto(110.60551524,97.34202056)(87.80829432,123.46874971)(86.60541544,151.3023715)
\curveto(84.90911321,190.55340136)(120.11646247,220.54254289)(157.7280043,221.11690252)
\curveto(208.12141186,221.8864519)(246.14100889,176.44484303)(245.87297704,127.87544245)
\curveto(245.55309543,69.91047155)(196.30227607,24.23197281)(139.99999909,19.99999773)
}
}
{
\pscustom[linestyle=none,fillstyle=solid,fillcolor=white]
{
\newpath
\moveto(10,270)
\lineto(10,10)
\lineto(270,10)
\lineto(270,270)
\lineto(10,270)
\closepath
\moveto(140,260)
\curveto(206.27417,260)(260,206.27417)(260,140)
\curveto(260,73.72583)(206.27417,20)(140,20)
\curveto(73.72583,20)(20,73.72583)(20,140)
\curveto(20,206.27417)(73.72583,260)(140,260)
\closepath
\moveto(370,270)
\lineto(370,10)
\lineto(630,10)
\lineto(630,270)
\lineto(370,270)
\closepath
\moveto(500,260)
\curveto(566.27417,260)(620,206.27417)(620,140)
\curveto(620,73.72583)(566.27417,20)(500,20)
\curveto(433.72583,20)(380,73.72583)(380,140)
\curveto(380,206.27417)(433.72583,260)(500,260)
\closepath
}
}
{
\pscustom[linewidth=1,linecolor=black]
{
\newpath
\moveto(260,139.99999738)
\curveto(260,73.7258274)(206.27416998,19.99999738)(140,19.99999738)
\curveto(73.72583002,19.99999738)(20,73.7258274)(20,139.99999738)
\curveto(20,206.27416736)(73.72583002,259.99999738)(140,259.99999738)
\curveto(206.27416998,259.99999738)(260,206.27416736)(260,139.99999738)
\closepath
\moveto(620,139.99999477)
\curveto(620,73.72582479)(566.27416998,19.99999477)(500,19.99999477)
\curveto(433.72583002,19.99999477)(380,73.72582479)(380,139.99999477)
\curveto(380,206.27416475)(433.72583002,259.99999477)(500,259.99999477)
\curveto(566.27416998,259.99999477)(620,206.27416475)(620,139.99999477)
\closepath
}
}
{
\pscustom[linewidth=1,linecolor=black,fillstyle=solid,fillcolor=white]
{
\newpath
\moveto(180,139.99999738)
\curveto(180,117.90860739)(162.09138999,99.99999738)(140,99.99999738)
\curveto(117.90861001,99.99999738)(100,117.90860739)(100,139.99999738)
\curveto(100,162.09138738)(117.90861001,179.99999738)(140,179.99999738)
\curveto(162.09138999,179.99999738)(180,162.09138738)(180,139.99999738)
\closepath
\moveto(540,139.99999477)
\curveto(540,117.90860477)(522.09138999,99.99999477)(500,99.99999477)
\curveto(477.90861001,99.99999477)(460,117.90860477)(460,139.99999477)
\curveto(460,162.09138476)(477.90861001,179.99999477)(500,179.99999477)
\curveto(522.09138999,179.99999477)(540,162.09138476)(540,139.99999477)
\closepath
}
}
{
\pscustom[linewidth=2,linecolor=black]
{
\newpath
\moveto(290,140)
\lineto(350,140)
}
}
{
\pscustom[linewidth=2,linecolor=black,fillstyle=solid,fillcolor=black]
{
\newpath
\moveto(330,140)
\lineto(322,132)
\lineto(350,140)
\lineto(322,148)
\lineto(330,140)
\closepath
}
}
{
\pscustom[linestyle=none,fillstyle=solid,fillcolor=black]
{
\newpath
\moveto(537,25)
\curveto(537,22.23857625)(534.76142375,20)(532,20)
\curveto(529.23857625,20)(527,22.23857625)(527,25)
\curveto(527,27.76142375)(529.23857625,30)(532,30)
\curveto(534.76142375,30)(537,27.76142375)(537,25)
\closepath
\moveto(145,100)
\curveto(145,97.23857625)(142.76142375,95)(140,95)
\curveto(137.23857625,95)(135,97.23857625)(135,100)
\curveto(135,102.76142375)(137.23857625,105)(140,105)
\curveto(142.76142375,105)(145,102.76142375)(145,100)
\closepath
\moveto(145,20)
\curveto(145,17.23857625)(142.76142375,15)(140,15)
\curveto(137.23857625,15)(135,17.23857625)(135,20)
\curveto(135,22.76142375)(137.23857625,25)(140,25)
\curveto(142.76142375,25)(145,22.76142375)(145,20)
\closepath
\moveto(505,100)
\curveto(505,97.23857625)(502.76142375,95)(500,95)
\curveto(497.23857625,95)(495,97.23857625)(495,100)
\curveto(495,102.76142375)(497.23857625,105)(500,105)
\curveto(502.76142375,105)(505,102.76142375)(505,100)
\closepath
\moveto(491,103)
\curveto(491,100.23857625)(488.76142375,98)(486,98)
\curveto(483.23857625,98)(481,100.23857625)(481,103)
\curveto(481,105.76142375)(483.23857625,108)(486,108)
\curveto(488.76142375,108)(491,105.76142375)(491,103)
\closepath
\moveto(505,20)
\curveto(505,17.23857625)(502.76142375,15)(500,15)
\curveto(497.23857625,15)(495,17.23857625)(495,20)
\curveto(495,22.76142375)(497.23857625,25)(500,25)
\curveto(502.76142375,25)(505,22.76142375)(505,20)
\closepath
}
}
{
\put(150,60){$A_K$}
\put(70,190){$A_S$}
}
\end{pspicture}
(b)
\psset{xunit=.5pt,yunit=.5pt,runit=.5pt}
\begin{pspicture}(640,280)
{
\pscustom[linestyle=none,fillstyle=solid,fillcolor=gray]
{
\newpath
\moveto(260,139.99999738)
\curveto(260,73.7258274)(206.27416998,19.99999738)(140,19.99999738)
\curveto(73.72583002,19.99999738)(20,73.7258274)(20,139.99999738)
\curveto(20,206.27416736)(73.72583002,259.99999738)(140,259.99999738)
\curveto(206.27416998,259.99999738)(260,206.27416736)(260,139.99999738)
\closepath
\moveto(620,139.99999477)
\curveto(620,73.72582479)(566.27416998,19.99999477)(500,19.99999477)
\curveto(433.72583002,19.99999477)(380,73.72582479)(380,139.99999477)
\curveto(380,206.27416475)(433.72583002,259.99999477)(500,259.99999477)
\curveto(566.27416998,259.99999477)(620,206.27416475)(620,139.99999477)
\closepath
}
}
{
\pscustom[linewidth=4,linecolor=black]
{
\newpath
\moveto(510.97794505,110.25426161)
\curveto(495.86579281,98.57074525)(474.45214036,106.27529803)(464.48685816,120.84973268)
\curveto(448.10902929,144.80265162)(460.29358332,176.89623493)(483.38010001,191.32480448)
\curveto(518.51598758,213.2839681)(563.94650624,195.36316792)(583.73461246,161.44377494)
\curveto(612.20056797,112.6494173)(587.20874259,51.07820321)(540.00000319,24.99999605)
\curveto(535.24444318,22.37301467)(530.30358075,20.08271863)(525.22896532,18.14283329)
\moveto(171.0975681,119.53234076)
\curveto(160.00815698,98.39541758)(131.97097916,94.79310677)(112.6385001,106.21278087)
\curveto(84.21612547,123.00184628)(79.82408412,161.6704012)(96.73198975,188.17335965)
\curveto(120.0683177,224.75280028)(170.84473134,229.97681756)(205.39708653,206.70729624)
\curveto(250.96876767,176.01673405)(257.06613622,111.73456899)(226.60987263,68.29329673)
\curveto(200.24933332,30.69395819)(152.52303749,13.46142106)(107.86922394,21.77647602)
}
}
{
\pscustom[linewidth=1.5,linecolor=gray]
{
\newpath
\moveto(161.52393849,108.10717774)
\curveto(141.00259469,91.96519312)(111.35320819,100.571345)(97.0048419,120.61836541)
\curveto(76.29822842,149.54889969)(88.52127689,189.64036811)(116.74578695,208.53036557)
\curveto(155.1941793,234.26296578)(207.08028836,217.88725908)(230.96979253,180.35260597)
\curveto(262.17280869,131.32712686)(241.13717569,66.36182773)(193.2162371,37.03140648)
\curveto(169.12200192,22.28432331)(140.03950071,16.85302824)(112.17128584,21.05209787)
\moveto(510.97794505,110.25426161)
\curveto(495.86579281,98.57074525)(474.45214036,106.27529803)(464.48685816,120.84973268)
\curveto(448.10902929,144.80265162)(460.29358332,176.89623493)(483.38010001,191.32480448)
\curveto(518.51598758,213.2839681)(563.94650624,195.36316792)(583.73461246,161.44377494)
\curveto(612.20056797,112.6494173)(587.20874259,51.07820321)(540.00000319,24.99999605)
\curveto(535.24444318,22.37301467)(530.30358075,20.08271863)(525.22896532,18.14283329)
}
}
{
\pscustom[linewidth=3,linecolor=black]
{
\newpath
\moveto(125,20)
\lineto(125,105)
\lineto(160,110)
\lineto(160,20)
\moveto(485,20)
\lineto(485,105)
\lineto(520,110)
\lineto(520,20)
}
}
{
\pscustom[linestyle=none,fillstyle=solid,fillcolor=white]
{
\newpath
\moveto(10,270)
\lineto(10,10)
\lineto(270,10)
\lineto(270,270)
\lineto(10,270)
\closepath
\moveto(140,260)
\curveto(206.27417,260)(260,206.27417)(260,140)
\curveto(260,73.72583)(206.27417,20)(140,20)
\curveto(73.72583,20)(20,73.72583)(20,140)
\curveto(20,206.27417)(73.72583,260)(140,260)
\closepath
\moveto(370,270)
\lineto(370,10)
\lineto(630,10)
\lineto(630,270)
\lineto(370,270)
\closepath
\moveto(500,260)
\curveto(566.27417,260)(620,206.27417)(620,140)
\curveto(620,73.72583)(566.27417,20)(500,20)
\curveto(433.72583,20)(380,73.72583)(380,140)
\curveto(380,206.27417)(433.72583,260)(500,260)
\closepath
}
}
{
\pscustom[linewidth=1,linecolor=black]
{
\newpath
\moveto(260,139.99999738)
\curveto(260,73.7258274)(206.27416998,19.99999738)(140,19.99999738)
\curveto(73.72583002,19.99999738)(20,73.7258274)(20,139.99999738)
\curveto(20,206.27416736)(73.72583002,259.99999738)(140,259.99999738)
\curveto(206.27416998,259.99999738)(260,206.27416736)(260,139.99999738)
\closepath
\moveto(620,139.99999477)
\curveto(620,73.72582479)(566.27416998,19.99999477)(500,19.99999477)
\curveto(433.72583002,19.99999477)(380,73.72582479)(380,139.99999477)
\curveto(380,206.27416475)(433.72583002,259.99999477)(500,259.99999477)
\curveto(566.27416998,259.99999477)(620,206.27416475)(620,139.99999477)
\closepath
}
}
{
\pscustom[linewidth=1,linecolor=black,fillstyle=solid,fillcolor=white]
{
\newpath
\moveto(180,139.99999738)
\curveto(180,117.90860739)(162.09138999,99.99999738)(140,99.99999738)
\curveto(117.90861001,99.99999738)(100,117.90860739)(100,139.99999738)
\curveto(100,162.09138738)(117.90861001,179.99999738)(140,179.99999738)
\curveto(162.09138999,179.99999738)(180,162.09138738)(180,139.99999738)
\closepath
\moveto(540,139.99999477)
\curveto(540,117.90860477)(522.09138999,99.99999477)(500,99.99999477)
\curveto(477.90861001,99.99999477)(460,117.90860477)(460,139.99999477)
\curveto(460,162.09138476)(477.90861001,179.99999477)(500,179.99999477)
\curveto(522.09138999,179.99999477)(540,162.09138476)(540,139.99999477)
\closepath
}
}
{
\pscustom[linewidth=2,linecolor=black]
{
\newpath
\moveto(290,140)
\lineto(350,140)
}
}
{
\pscustom[linewidth=2,linecolor=black,fillstyle=solid,fillcolor=black]
{
\newpath
\moveto(330,140)
\lineto(322,132)
\lineto(350,140)
\lineto(322,148)
\lineto(330,140)
\closepath
}
}
{
\pscustom[linestyle=none,fillstyle=solid,fillcolor=black]
{
\newpath
\moveto(129,105)
\curveto(129,102.790861)(127.209139,101)(125,101)
\curveto(122.790861,101)(121,102.790861)(121,105)
\curveto(121,107.209139)(122.790861,109)(125,109)
\curveto(127.209139,109)(129,107.209139)(129,105)
\closepath
\moveto(164,105)
\curveto(164,102.790861)(162.209139,101)(160,101)
\curveto(157.790861,101)(156,102.790861)(156,105)
\curveto(156,107.209139)(157.790861,109)(160,109)
\curveto(162.209139,109)(164,107.209139)(164,105)
\closepath
\moveto(129,21)
\curveto(129,18.790861)(127.209139,17)(125,17)
\curveto(122.790861,17)(121,18.790861)(121,21)
\curveto(121,23.209139)(122.790861,25)(125,25)
\curveto(127.209139,25)(129,23.209139)(129,21)
\closepath
\moveto(164,20)
\curveto(164,17.790861)(162.209139,16)(160,16)
\curveto(157.790861,16)(156,17.790861)(156,20)
\curveto(156,22.209139)(157.790861,24)(160,24)
\curveto(162.209139,24)(164,22.209139)(164,20)
\closepath
\moveto(489,103)
\curveto(489,100.790861)(487.209139,99)(485,99)
\curveto(482.790861,99)(481,100.790861)(481,103)
\curveto(481,105.209139)(482.790861,107)(485,107)
\curveto(487.209139,107)(489,105.209139)(489,103)
\closepath
\moveto(477,110)
\curveto(477,107.790861)(475.209139,106)(473,106)
\curveto(470.790861,106)(469,107.790861)(469,110)
\curveto(469,112.209139)(470.790861,114)(473,114)
\curveto(475.209139,114)(477,112.209139)(477,110)
\closepath
\moveto(466,134)
\curveto(466,131.790861)(464.209139,130)(462,130)
\curveto(459.790861,130)(458,131.790861)(458,134)
\curveto(458,136.209139)(459.790861,138)(462,138)
\curveto(464.209139,138)(466,136.209139)(466,134)
\closepath
\moveto(524,105)
\curveto(524,102.790861)(522.209139,101)(520,101)
\curveto(517.790861,101)(516,102.790861)(516,105)
\curveto(516,107.209139)(517.790861,109)(520,109)
\curveto(522.209139,109)(524,107.209139)(524,105)
\closepath
\moveto(489,21)
\curveto(489,18.790861)(487.209139,17)(485,17)
\curveto(482.790861,17)(481,18.790861)(481,21)
\curveto(481,23.209139)(482.790861,25)(485,25)
\curveto(487.209139,25)(489,23.209139)(489,21)
\closepath
\moveto(524,22)
\curveto(524,19.790861)(522.209139,18)(520,18)
\curveto(517.790861,18)(516,19.790861)(516,22)
\curveto(516,24.209139)(517.790861,26)(520,26)
\curveto(522.209139,26)(524,24.209139)(524,22)
\closepath
\moveto(546,28)
\curveto(546,25.790861)(544.209139,24)(542,24)
\curveto(539.790861,24)(538,25.790861)(538,28)
\curveto(538,30.209139)(539.790861,32)(542,32)
\curveto(544.209139,32)(546,30.209139)(546,28)
\closepath
\moveto(562,33)
\curveto(562,30.790861)(560.209139,29)(558,29)
\curveto(555.790861,29)(554,30.790861)(554,33)
\curveto(554,35.209139)(555.790861,37)(558,37)
\curveto(560.209139,37)(562,35.209139)(562,33)
\closepath
}
}
\end{pspicture}
\caption{\label{pic2}}
\end{figure}
When there are instead multiple components to consider, we treat all the parallel copies as a product neighbourhood of a single copy, and isotope this neighbourhood as we just described for one annulus (see Figure \ref{pic2}b).

Similarly, in each $E'_j\setminus W'$, we will treat all parallel copies of a connected component as a product neighbourhood of one component. Again, therefore, picture the case where $R_K\cap(E'_k\setminus W')$ has a single component (the same will then be true of $S\cap(E'_k\setminus W')$). Denote the component of $R_K$ by $B_K$ and the component of $S$ by $B_S$. 

Initially $B_K$ and $B_S$ coincide. We have already defined our isotopy on $\partial(E'_k\setminus W')$; we wish to extend this isotopy to $E'_k\setminus W'$ so that $B_S$ becomes transverse to $B_K$. 
Each boundary component has a pre-defined direction that it needs to be moved. 
Since $B_S$ is connected, a suitable isotopy of the surface $B_S$ exists after which $|B_K\cap B_S|\leq 1$, with $B_K$ and $B_S$ disjoint if possible. 

We can see more explicitly how the isotopy is chosen as follows. Each boundary component of $B_S$ needs to be moved either in the direction given by the orientation on $B_K$ or in the opposite direction. Mark a boundary component with a $+$ if the direction it needs to move agrees with the orientation of $B_K$, and with a $-$ otherwise. These signs can also be determined using the coefficients $s_i$. Recall that $B_S$ lies outside $T'_k$. If $s_k>0$ then the boundary component of $B_S$ on $\partial W'_k$ has a $-$ sign, and if $s_k<0$ then it has a $+$ sign. If $B_S$ has a boundary component on $\partial W'_j$ for some $j\neq k$ then $B_S$ lies inside $T'_j$. If $s_j>0$ then this boundary component has a $+$, and if $s_j<0$ then it has a $-$. If all boundary components of $B_S$ have the same sign, we may isotope $B_S$ to be disjoint from $B_K$ in a way that behaves as required on the boundary. Otherwise, choose a single simple closed curve on $B_S$ that separates all boundary components with a $+$ from all boundary components with a $-$. In this case we can choose a suitable isotopy that leaves this curve as the intersection between $B_K$ and $B_S$. The section of $B_S$ on the $+$ side of the curve is isotoped to the positive side of $B_K$, while the section on the $-$ side is isotoped to the negative side of $B_K$.

There is one case not included in this description. If $k=0$ then $B_S$ has one boundary component on $T_0$. There is no pre-determined position we must isotope this boundary component to. To avoid creating product regions, if all other boundary components have the same sign then we must also assign that sign to this boundary component. Otherwise, we may freely assign it either a $+$ or a $-$.

We have now made $R_K$ and $S$ transverse by an isotopy of $S$ in $E$. To apply Proposition \ref{prop:productregions}, we must verify that our choice of isotopy was a good one, that there are now no product regions bounded by $R_K$ and $S$. We must therefore check each of the complementary regions of $E\setminus(R_K\cup S)$ to see if it is a product region. 

The first thing to note is that any complementary region that lies between two parallel copies of a section of $R_K$ will meet $R_K$ on both its positive and negative sides, since we have chosen $R_K$ such that all such sections of surface are oriented in the same direction. Therefore, these complementary regions cannot be product regions. The same holds for parallel sections of $S$. See Figure \ref{pic3}a; the shaded regions pick out one complementary region between parallel sections of $R_K$, one between parallel sections of $S$, and one coming from the intersection of parallel regions that therefore lies both between parallel sections of $R_K$ and between parallel sections of $S$.
\begin{figure}[htb]
\centering
(a)
\psset{xunit=.5pt,yunit=.5pt,runit=.5pt}
\begin{pspicture}(270,330)
{
\pscustom[linestyle=none,fillstyle=solid,fillcolor=gray]
{
\newpath
\moveto(180,200.00000262)
\lineto(150,200.00000262)
\lineto(150,170.00000262)
\lineto(180,170.00000262)
\closepath
\moveto(20,170.00000262)
\lineto(20,140.00000262)
\lineto(120,140.00000262)
\lineto(120,170.00000262)
\closepath
\moveto(150,310.00000262)
\lineto(120,310.00000262)
\lineto(120,200.00000262)
\lineto(150,200.00000262)
\closepath
}
}
{
\pscustom[linewidth=4,linecolor=black]
{
\newpath
\moveto(150,120.00000262)
\lineto(150,310.00000262)
\moveto(120,120.00000262)
\lineto(120,310.00000262)
\moveto(180,120.00000262)
\lineto(180,310.00000262)
}
}
{
\pscustom[linewidth=2,linecolor=black]
{
\newpath
\moveto(210,169.99998262)
\lineto(20,169.99998262)
\moveto(210,139.99998262)
\lineto(20,139.99998262)
\moveto(210,199.99998262)
\lineto(20,199.99998262)
}
}
{
\pscustom[linewidth=1,linecolor=black]
{
\newpath
\moveto(180,300.00000262)
\lineto(200,300.00000262)
\moveto(150,300.00000262)
\lineto(170,300.00000262)
\moveto(120,300.00000262)
\lineto(140,300.00000262)
\moveto(30,199.99998262)
\lineto(30,219.99998262)
\moveto(30,169.99998262)
\lineto(30,189.99998262)
\moveto(30,139.99998262)
\lineto(30,159.99998262)
}
}
{
\pscustom[linewidth=1,linecolor=black,fillstyle=solid,fillcolor=black]
{
\newpath
\moveto(190,300.00000262)
\lineto(186,296.00000262)
\lineto(200,300.00000262)
\lineto(186,304.00000262)
\lineto(190,300.00000262)
\closepath
\moveto(160,300.00000262)
\lineto(156,296.00000262)
\lineto(170,300.00000262)
\lineto(156,304.00000262)
\lineto(160,300.00000262)
\closepath
\moveto(130,300.00000262)
\lineto(126,296.00000262)
\lineto(140,300.00000262)
\lineto(126,304.00000262)
\lineto(130,300.00000262)
\closepath
\moveto(30,209.99998262)
\lineto(34,205.99998262)
\lineto(30,219.99998262)
\lineto(26,205.99998262)
\lineto(30,209.99998262)
\closepath
\moveto(30,179.99998262)
\lineto(34,175.99998262)
\lineto(30,189.99998262)
\lineto(26,175.99998262)
\lineto(30,179.99998262)
\closepath
\moveto(30,149.99998262)
\lineto(34,145.99998262)
\lineto(30,159.99998262)
\lineto(26,145.99998262)
\lineto(30,149.99998262)
\closepath
}
}
{
\put(220,160){$S$}
\put(130,100){$R_K$}
}
\end{pspicture}
(b)
\input{pic3b.tex}
\caption{\label{pic3}}
\end{figure}

Hence once more we can imagine that each part of each of the surfaces $R_K$ and $S$ has a single component rather than multiple parallel copies of a component.
Observe that, under this assumption, our choice of sign for the boundary component of $S$ on $T_0$ ensures that every complementary region meets $T'_j$ for some $j\in\{1,\ldots,N'\}$.

Next we turn our attention to the complementary regions that are contained entirely within $ W'$. This is depicted in Figure \ref{pic3}b.
Again we find that each such complementary region (such as that marked $M_1$ in Figure \ref{pic3}b) meets $S$ on both the positive and negative sides. This is also true of any complementary region that intersects $E'_j\cap\partial  W'$ in a `small' sub-annulus of $E'_j\cap\partial  W'$ across which $S$ was isotoped (such as that marked $A_1$ in Figure \ref{pic3}b). 

The final case to consider is a complementary region that intersects $E'_j\cap\partial W'$ in a `larger' sub-annulus coming from a component of $(E'_j\cap\partial W')\setminus R_K$ (such as that marked $A_2$ in Figure \ref{pic3}b). There are two possibilities. If $R_K$ and $S$ intersect in $E'_j\setminus W'$ (that is, if there are boundary components of $S\cap E'_j$ that were marked with different signs) then once again the complementary region meets both $R_K$ and $S$ each on the positive side and on the negative side. The other possibility is that $S\cap E'_j$ and $R_K\cap E'_j$ are disjoint and parallel in $E'_j\setminus  W'$. This time we cannot necessarily use the boundary pattern to rule out the possibility that the complementary region of interest is a product region. However, the fact that $R_K\cap E'_j$ is not a fibre for $E'_j$ tells us this instead.

As $R_K$ and $S$ do not bound a product region, Proposition \ref{prop:productregions} allows us to use $S$ to calculate the distance between $R_K$ and $S$ in $\ms(K)$ without any further isotopy. Choose $k\in\{1,\ldots,N'\}$ such that $s_k=\max(|s_1|,\ldots,|s_{N'}|)$.
An annulus of $S\cap W'_k$ contains at least $5|s_k|-1$ curves of intersection with $R_K$, each a core curve of the annulus and all oriented in the same direction. Therefore, in the cover $\widetilde{E}$ of $E$ constructed using $R_K$, a lift of $S$ intersects at least $5|s_k|$ lifts of $E\setminus R_K$. Hence $d_{\ms(K)}(R_K,S)\geq 5|s_k|=5\max(|s_1|,\ldots,|s_{N'}|)=5\max(|s_1-r_1|,\ldots,|s_{N'}-r_{N'}|)$.

We can also find an upper bound on $d_{\ms(K)}(R_K,S)$, since this is at most $|R_K\cap S|+1$. Set $M=\max (|R_K\cap T'_1|,\ldots,|R_K\cap T'_{N'}|)$. Then
\[
|R_K\cap S|\leq M^2(5\max(|s_1|,\ldots,|s_{N'}|)-1)+(N'+1),
\]
so
\[
\dist_{\ms(K)}(R_K,S)\leq 5M^2\max(|s_1-r_1|,\ldots,|s_{N'}-r_{N'}|) +(N'+2).
\]
These two inequalities together show that $\Theta$ is a quasi-isometric embedding.
\end{proof}

\begin{corollary}
The upper bound on dimension given in \cite{zbMATH06369090} is also a lower bound.
That is, for a knot $K$ in $\Sphere$, the Kakimizu complex $\ms(K)$ of $K$ is quasi-isometric to $\mathbb{Z}^M$, where $M$ is equal to one less than the number of core JSJ blocks that are not fibred.
\end{corollary}

\subsection*{A word on links}

The results in this paper, like those in \cite{zbMATH06369090}, are specifically stated for knots, rather than links in general. 
The definition of the Kakimizu complex and the metric on it can be extended to links. However, the definitions should be stated in a different form before generalising. For more details on this see \cite{MR2869183} and \cite{MR3091274}.

The reason for the restriction to knots comes in Theorem 7 of \cite{zbMATH06369090}, which shows that there are only finitely many subsurfaces in each block that are relevant for the main proof.  
This is proved using the classification of Seifert fibred submanifolds of \(\Sphere\) given by Budney in \cite{zbMATH05141213}. Although Budney's result applies equally well for multi-component link complements as for knot complements, the same is not true of \cite{zbMATH06369090} Theorem 7.

As a counter-example, consider the \((9,6)\) torus link (that is, three parallel copies of a trefoil), with all components oriented in parallel. The complement of this link is Seifert fibred over a punctured sphere with two exceptional fibres. Since there is therefore only one block in the link complement, we would want to conclude that the Kakimizu complex is quasi-isomorphic to a point. On the other hand, there is an essential torus in the link complement separating two of the link components from the third, which can be used for spinning around.
In calculating the dimension of the Kakimizu complex of a link complement, it is thus important to allow for the presence of toroidal Seifert fibred pieces in the JSJ decomposition.


\bibliography{quasirefs}

\begin{thebibliography}{1}

\bibitem{zbMATH05899716}
Jessica~E. {Banks}.
\newblock {On links with locally infinite Kakimizu complexes.}
\newblock {\em {Algebr. Geom. Topol.}}, 11(3):1445--1454, 2011.

\bibitem{MR3091274}
Jessica~E. Banks.
\newblock The {K}akimizu complex of a connected sum of links.
\newblock {\em Trans. Amer. Math. Soc.}, 365(11):6017--6036, 2013.

\bibitem{zbMATH05141213}
Ryan {Budney}.
\newblock {JSJ-decompositions of knot and link complements in $S^3$.}
\newblock {\em {Enseign. Math. (2)}}, 52(3-4):319--359, 2006.

\bibitem{burde2013knots}
Gerhard Burde, Heiner Zieschang, and Michael Heusener.
\newblock {\em Knots}, volume~5 of {\em De Gruyter Studies in Mathematics}.
\newblock De Gruyter, Berlin, extended edition, 2014.

\bibitem{zbMATH06369090}
Jesse {Johnson}, Roberto {Pelayo}, and Robin {Wilson}.
\newblock {The coarse geometry of the Kakimizu complex.}
\newblock {\em {Algebr. Geom. Topol.}}, 14(5):2549--2560, 2014.

\bibitem{MR1177053}
Osamu Kakimizu.
\newblock Finding disjoint incompressible spanning surfaces for a link.
\newblock {\em Hiroshima Math. J.}, 22(2):225--236, 1992.

\bibitem{MR2869183}
Piotr Przytycki and Jennifer Schultens.
\newblock Contractibility of the {K}akimizu complex and symmetric {S}eifert
  surfaces.
\newblock {\em Trans. Amer. Math. Soc.}, 364(3):1489--1508, 2012.

\end{thebibliography}
\bibliographystyle{hplain}


\bigskip

\noindent
University of Hull

\noindent
Hull, HU6 7RX

\noindent
UK

\smallskip
\noindent
\textit{jessica.banks[at]lmh.oxon.org/j.banks[at]hull.ac.uk}

\end{document}